\documentclass[a4paper]{article}

\usepackage{amsmath}
\usepackage{amsthm}
\usepackage{amssymb}

\newtheorem{Theorem}{Theorem}
\newtheorem{Lemma}[Theorem]{Lemma}

\newcommand\mR{\mathbb R}
\newcommand\mE{\mathbb E}
\newcommand\ep\varepsilon
\makeatletter
\newcommand{\raisemath}[1]{\mathpalette{\raisem@th{#1}}}
\newcommand{\raisem@th}[3]{\raisebox{#1}{$#2#3$}}
\makeatother

\date{} 
\title{A polarization identity for multilinear
maps\footnote{This elegant note belongs to the legacy of Erik Thomas,
who conceived and wrote it in 1997 after discussions on functional
integration with his former student J. Leo van Hemmen (TU M\"unchen).
It has been edited to some extent by the latter and Tom H. Koornwinder
(University of Amsterdam) but is essentially as it was when the author
passed away on 13~September, 2011. The result presented in this paper
has turned out to date back to Mazur and Orlicz \cite[Eq.~(22)]{mo34} but the proof
by Thomas is
far simpler and incomparably shorter, making its arguments attractive
and justifying a separate publication. The Appendix, presenting a further
simplification which has appeared first in \cite{bs71}, was written by Tom H.
Koorninder in 2013.}} 

\author{Erik G.F. Thomas\footnote{With an Appendix by Tom H. Koornwinder. Correspondence to: Korteweg-de Vries Institute, University of Amsterdam, P.O.\ Box 94248, 1090 GE
\mbox{Amsterdam}, The Netherlands.
e-mail: {\tt T.H.Koornwinder@uva.nl}.}\\[\medskipamount]
{\footnotesize\em Johann Bernoulli Institute for
Mathematics and Computer Science},\\
{\footnotesize\em University of \mbox{Groningen},
P.O.\ Box 407, 9700 AK Groningen, The Netherlands}}

\begin{document}
\maketitle

\begin{abstract}
Given linear spaces $E$ and $F$ over the real numbers or a field of
characteristic zero, a simple argument is given to represent a
symmetric multilinear map $u(x_{1}, x_{2}, \ldots, x_{n})$ from $E^n$
to $F$ in terms of its restriction to the diagonal. As an application,
a probabilistic expression for Gaussian variables used by Nelson and
by Schetzen is derived. An Appendix by Tom H. Koornwinder notes an
even further simplification by Bochnak and Siciak (1971) of the proof of
the main result.\end{abstract}

\section{Introduction}

A symmetric bilinear form $(x_1,x_2) \mapsto u(x_1,x_2)$ is well known
to be completely determined by the corresponding quadratic form $
\tilde u(x) = u(x,x).$ For instance, $u$ is determined by $\tilde u$
through any of the following formulas, known as polarization
identities,
\begin{subequations} 
\begin{align}  
u(x_1,x_2) &= \frac{1}{4} \left[ \tilde{u}(x_1+x_2) - 
\tilde{u}(x_1-x_2 )\right],
\label{eq:1a}\\
u(x_1,x_2) &= \frac{1}{2} \left[ \tilde{u}(x_1+x_2) - \tilde{u}(x_1) - 
\tilde{u}(x_2) \right].
\label{eq:1b}
\end{align}
\end{subequations}
Given linear spaces $E$ and $F$ over the real number field $\mR$ or,
more generally, over a field $K$ of characteristic zero, this paper
proposes a similar formula to express an $n$-linear map
$u:E^n\longrightarrow F$ which is symmetric, i.e., which is invariant
under all $n!$ permutations of the variables, in terms of its
restriction to the main diagonal
\begin{equation*}
\tilde u(x) = u(x,\ldots ,x).
\end{equation*}
A formula which accomplishes this, such as \eqref{eq:1a} or
\eqref{eq:1b}, will be called a {\sl polarization identity}. In the
works of E. Nelson \cite{nelson} and M. Schetzen
\cite{schetzen}, in different ways concerned with products of Gaussian
random variables, a general polarization identity is given, without
the ``obvious'' combinatorial proof. In the present paper we give a
proof based on simple properties of the shift operator.

\section{A polarization identity using operators}

We define two operators acting on functions $v:E\longrightarrow F$.
A difference operator $\Delta _h,$ depending on $h\in E,$ and Tr, the 
trace or value at the origin:
\begin{equation*}
\left( \Delta_h v \right)(x) = v(x+h) - v(x),\qquad
\text{\rm Tr} \;v = v(0).
\end{equation*}
\begin{Theorem}
Let $E$ and $F$ be linear spaces over a field $K$ of characteristic 
zero, and let $u:E^n\longrightarrow F$ be a symmetric $n$-linear map.
Then we have the polarization identity
\begin{equation} \label{eq:4}
 u(x_1,\ldots ,x_n) = \frac{1}{n!} \; \text{\rm Tr}\; \Delta 
_{x_n}\Delta_{x_{n-1}}\ldots \Delta _{x_1}\tilde u.
\end{equation} 
\end{Theorem}
For $n=2$ this is the formula \eqref{eq:1b}. For $n=3$ we obtain
\begin{eqnarray*}
u(x_1,x_2,x_3) = \frac{1}{6} \left[ \tilde{u}(x_1 + x_2 + x_3) 
-\tilde{u}(x_1 + x_2) - \tilde{u}(x_2 + x_3) - \tilde{u}(x_3 + x_1) 
\right.
\\
\left.
+ \tilde{u}(x_1) + \tilde{u}(x_2) + \tilde{u}(x_3) \right].
\end{eqnarray*}
The general case will be worked out below.
\begin{proof} 
We shall prove quite generally the following relation, which in the 
case of characteristic zero is equivalent to \eqref{eq:4},
\begin{equation} \label{eq:5}
n!\,u(x_1,\ldots ,x_n) = \text{\rm Tr}\; \Delta 
_{x_n}\Delta_{x_{n-1}}\ldots\Delta _{x_1}\tilde u.
\end{equation}
It is sufficient to prove this in the case of finite dimensional
spaces $E$ and $F$, since we can replace $E$ by the subspace spanned
by $x_1, \ldots , x_n$, and then $F$ by the linear span of the image
$u(E^n)$. In the remainder of the proof we assume that the spaces $E$
and $F$ are finite dimensional. In the case of the real number field
($K=\mR$), the following formula, which of course is the motivation
behind \eqref{eq:4}, is well known:
\begin{equation} \label{eq:6}
u(x_1,\ldots ,x_n) = \frac{1}{n!} {\frac{\partial^n}{\partial 
t_1\partial t_2 \ldots \partial t_n}}\,\tilde u(t_1x_1+\cdots+ t_n x_n).
\end{equation}
The foregoing expression holds, not just at $(t_1,\ldots ,t_n) = 0$
but identically, in $t_1,t_2,\ldots ,t_n \in \mR$. To prove this in the
case of the real number field $\mR $ one can make use of the `product
rule' for differentiation: if $\varphi _1,\ldots ,\varphi _n$ are
differentiable functions on $\mR $ with values in $E$, one has
\begin{multline*}
\frac d{dt}u(\varphi _1(t),\ldots ,\varphi _n(t))\\
=u(\varphi _1'(t),\varphi _2(t),\ldots ,\varphi _n(t)) + \cdots
   + u(\varphi _1(t),\ldots,\varphi _{n-1}(t), \varphi _n'(t)).
\end{multline*}
Differentiating $\tilde u(t_1x_1+\ldots +t_nx_n)$ successively with
respect to $t_1,t_2,\ldots ,t_n$ one obtains formula \eqref{eq:6}. 

To prove \eqref{eq:6} for a general field $K$ of arbitrary
characteristic, we use the abbreviation $\partial ^n = \partial
^n / \partial t_1\partial t_2\ldots \partial t_n$. It is clear that 
for every monomial $t^\alpha = t_1^{\alpha _1}t_2^{\alpha _2}\ldots
t_n^{\alpha _n}$, with $\alpha _1+\cdots +\alpha _n = n$, we have
$\partial ^nt^\alpha = 0$, except for $t_1t_2\ldots t_n$ where we have
$\partial ^nt_1\ldots t_n = 1$. Therefore $\partial ^n\tilde
u(t_1x_1+\cdots +t_nx_n)$ equals the coefficient of the monomial
$t_1\ldots t_n$ in the development of $\tilde u(t_1x_1+\cdots
+t_nx_n)$. This is the sum of the term $u(x_1,\ldots ,x_n)$ and of
other similar terms obtained by permuting the $x_i$. Since there are
$n!$ such permutations, and $u$ is symmetric, we have
\begin{equation} \label{eq:7}
n!\, u(x_1,\ldots ,x_n) = \frac{\partial ^n}{\partial t_1\partial t_2
\ldots \partial t_n} \tilde u(t_1x_1+\cdots +t_nx_n).
\end{equation}
In the case of characteristic zero this is equivalent to \eqref{eq:6}.
Note that this already implies the uniqueness result: {\sl Two 
symmetric $n$-linear forms $u$ and $v$ on $E$ are equal if and only 
if $u(x,\ldots ,x) = v(x,\ldots ,x)$ for all $x\in E$}.
In the case of the real number field we obtain formula \eqref{eq:5} 
by integrating the expression \eqref{eq:7} $n$ times between $0$ and 
$1$:
\begin{align}
n!\,u(x_1,\ldots ,x_n) &=\int _0^1 \frac\partial{\partial 
t_1}\frac{\partial^{n-1}}{\partial t_2\ldots\partial t_n}\tilde 
u(t_1x_1+t_2x_2+\cdots +t_nx_n)dt_1 \nonumber\\
&=\frac{\partial^{n-1}}{\partial t_2\ldots\partial t_n}\big(\tilde 
u(x_1+t_2x_2+\cdots+ t_nx_n)-\tilde u(t_2x_2+\cdots +t_nx_n)\big) 
\nonumber \\
&=\frac{\partial^{n-1}}{\partial t_2\ldots\partial t_n}\Delta 
_{x_1}\tilde u(t_2x_2+\cdots +t_nx_n) \nonumber \\
&=\frac{\partial^{n-1}}{\partial t_3\ldots\partial t_n}\Delta 
_{x_2}\Delta _{x_1}\tilde u(t_3x_3+\cdots+ t_nx_n) =\ldots \nonumber\\
&=\frac\partial{\partial t_n}\big(\Delta _{x_{n-1}}\ldots 
\Delta_{x_1}\tilde u\big)(t_nx_n) \nonumber\\
&=\big(\Delta _{x_n}\Delta _{x_{n-1}}\ldots \Delta _{x_1}\tilde 
u\big)(0) \,. \label{eq:8} 
\end{align}
This completes the proof.\end{proof}
In the above proof we have not used the Riemann integral in an essential
way. Instead of integrating we can use the following lemma. 
\begin{Lemma}
Let $P$ be a polynomial in one variable. Then, if $P' = c$ is 
constant, $c = P(1)-P(0)$.
\end{Lemma}
\begin{proof} 
$P(t) = a_0+a_1t+\cdots +a_nt^n,$ with $a_i \in F$.
Then $P'(t) = a_1+2a_2t +\cdots +na_nt^{n-1} = c$.
Thus $a_1 = c$ and $a_2=\ldots =a_n = 0$.
Therefore $P(t) = a_0+ct$ and $P(1)-P(0) = c$. 
\end{proof}
By applying the lemma $n$ times it follows that \eqref{eq:8}, with the
right-hand side of the first line removed, is still valid in the case
of an arbitrary field. In conclusion, formula \eqref{eq:5} is valid
quite generally, i.e., for linear spaces $E$ and $F$ over an arbitrary
field and a symmetric $n$-linear map $u:E^n\longrightarrow F$. In
particular this proves the theorem. It also proves that the theorem is
false if the field has a positive characteristic dividing $n$.

\section{The formula worked out explicitly}

In the work of E. Nelson \cite{nelson} and M. Schetzen \cite{schetzen}
a formula is given, (3) resp. (5.4-10), for the product of $n$
numbers, which suggests the following polarization
formula:
\begin{equation}
n!\,u(x_1,\ldots,x_n)=\sum_{k=1}^n(-1)^{n-k}\sum_{J;\,|J|=k}\tilde u(S_J).
\label{eq:9}
\end{equation}
Here the inner sum on the right-hand side runs over subsets
$J$ of  $\{1,2,\ldots ,n\}$ for which $|J|$ (the number of elements in $J$)
is equal to $k$. We have also used the abbreviation
\begin{equation} \label{eq:12}
S_J = \sum_{i\in J}x_i \,.
\end{equation}

We assert that formula \eqref{eq:9} is just \eqref{eq:5} developed
explicitly. For the proof we introduce the shift operator $\sigma _h$
defined for functions $v: E \longrightarrow F$ by
\begin{equation*}
(\sigma _hv)(x) = v(x+h).
\end{equation*}
Thus $\Delta _h = \sigma _h - I,$ where $I = \sigma _0$ is the 
identity.
Note the rule:
\begin{equation} \label{eq:11}
\sigma_h\sigma _k = \sigma _{h+k}, \qquad h,k \in E.
\end{equation}
Then, since the operators $\sigma _{x_i}$ commute, we have 
\begin{equation}
\prod_{i=1}^n \Delta _{x_i} = \prod_{i=1}^n \big(\sigma _{x_i}-I\big) 
= \sum_J (-1)^{n-|J|}\prod_{i\in J} \sigma_{x_i}\,.
\label{eq:13}
\end{equation}
Ordering the sum by the number $k = |J|$ of elements in $J$ and 
taking into account \eqref{eq:11} we obtain
\begin{equation*}
\prod_{i=1}^n \Delta _{x_i} = \sum _{k=0}^n(-1)^{n-k} \sum
_{|J|=k}\sigma_{S_J}\ .
\end{equation*}
Let both sides of this operator identity act on  $\tilde u$:
\begin{equation}
\left(\prod_{i=1}^n \Delta _{x_i} \tilde u\right)(x)=
 \sum _{k=0}^n(-1)^{n-k}\sum _{|J|=k}\left(\sigma_{S_J}\tilde u\right)(x).
 \label{eq:17}
\end{equation}
Then by putting $x=0$  and by applying \eqref{eq:5} we obtain \eqref{eq:9}.
%
%
This proves the assertion. In conclusion, formula \eqref{eq:9} like
formula \eqref{eq:5} is valid for arbitrary fields. In particular, in
the case of a field of characteristic zero, this gives a polarization
formula as in \eqref{eq:4} or as in \eqref{eq:9} after division by
$n!\,$. \\[\bigskipamount] {\bf Remark}\quad Nelson \cite{nelson}
stated the formula not just for numbers, but for `commuting
indeterminates', for instance elements in a commutative algebra
$\mathcal{A}$. In this generality the formula
\begin{equation*}
   a_1a_2\ldots a_n = \frac1{n!} \sum _{k=0}^n (-1)^{n-k}\sum
   _{i_1<\ldots <i_k}(a_{i_1}+\cdots +a_{i_k})^n
\end{equation*}
is a consequence of formula \eqref{eq:9}, if we take $E = F = 
\mathcal{A},$ and $u(a_1,\ldots ,a_n) = a_1\ldots a_n.$

\section{Gaussian variables}
The polarization formula can be used, as was done in Nelson's paper
\cite{nelson}, to prove the formula for the expectation of a product
of Gaussian stochastic variables: if $x_1,\ldots ,x_n$ are stochastic
variables on some probability space, with a joint distribution which
is centered Gaussian, the expectation $\mE(x_1\ldots x_n)$ of the
product equals zero if $n$ is odd, the joint distribution being
invariant under the symmetry $(x_1,\ldots ,x_n) \mapsto
(-x_1,-x_2,\ldots ,-x_n)$ . But if $n$ is even we have
\begin{equation} \label{eq:16}
\mE(x_1\ldots x_n) = \sum \prod \mE(x_ix_j).
\end{equation}
In this sum of products each product is associated to a partition of 
the set $\{1,\ldots,n\}$ into sets of two elements, the product 
ranging over all the sets $\{i,j\}$ of a given partition. 
The sum is then taken for all such partitions. 
For instance 
\begin{equation*}
\mE(x_1x_2x_3x_4) = \mE(x_1x_2) 
\mE(x_3x_4)+\mE(x_1x_3)E(x_2x_4)+\mE(x_1x_4)\mE(x_2x_3).
\end{equation*}
Now it is clear that both sides of equation \eqref{eq:16} represent
symmetric $n$-linear forms on the space $E$ spanned by the random
variables $x_i$. The joint distribution being Gaussian, all the
elements in $E$ have Gaussian distributions. Therefore, just by the
uniqueness property, or formula \eqref{eq:5}, it follows that it is
enough to prove the formula in the case $x_1=x_2=\ldots =x_n = x,$
where by the homogeneity we may assume $\mE(x^2) = 1$, i.e.,
\[
\mE(x^n)= \int_{-\infty}^\infty
x^n e^{-\frac12 x^2}\,dx\Big/ \int_{-\infty}^\infty x^2 e^{-\frac12
x^2}\,dx.
\]
Formula \eqref{eq:16} then
reduces to
\begin{equation*}
\mE(x^n) = \text{number of partitions of } \{1,\ldots ,n\} \text{ into 
sets of two elements}  \ .
\end{equation*}
But this is true, because both sides are easily seen to be equal to
the product $(n-1)(n-3)\ldots (1)$. In Wiener's lectures
\cite{wiener58}, the case where the $x_i$ are equal was treated in
Lecture~1, and the general case was handled in Lecture 2, Eq.~(2.6).
Wiener's claim that the general case follows from the special case
(``From Lecture~1, it follows that\ldots'') is therefore quite
justified. 
Wiener has also given more details of his proof elsewhere
\cite[pp.~538--539]{wiener76}. 

\section{Question} 
Formulas \eqref{eq:1a} and \eqref{eq:1b} show that polarization
identities, with $n$ given, are not unique but formula \eqref{eq:1b}
has a symmetry property not shared by \eqref{eq:1a}, namely for any
function $\tilde u$ the expression on the right-hand side is seen to
be symmetric, just by the commutativity of the addition. This is also
true for the polarization identity \eqref{eq:9}. Does this symmetry
property characterize the formula uniquely?

\numberwithin{equation}{section}
\appendix
\section{Appendix (by Tom H. Koornwinder)}
As pointed out in the footnote on the first page, formula \eqref{eq:9}
was first obtained by Mazur \& Orlicz \cite[(22)]{mo34}. They gave the
formula in the equivalent form
\begin{equation}
n!\,u(x_1,\ldots,x_n)=\sum_{\ep_1,\ldots,\ep_n=0}^1(-1)^{n-(\ep_1+\cdots\ep_n)}
\tilde u(\ep_1 x_1+\cdots+\ep_n x_n).
\label{K5}
\end{equation}
In fact, a more general formula than \eqref{eq:9} was given
by Bochnak \& Siciak \cite[(1)]{bs71},
\begin{equation}
n!\,u(x_1,\ldots,x_n)=\sum_{\ep_1,\ldots,\ep_n=0}^1(-1)^{n-(\ep_1+\cdots\ep_n)}
\tilde u(x_0+\ep_1 x_1+\cdots+\ep_n x_n)
\label{K1}
\end{equation}
with $x_0$ arbitrary. The authors \cite{bs71} attributed formula
(\ref{K1}) to Mazur and Orlicz \cite{mo34}, but I could not find it 
there. Note that \eqref{K1} can be equivalently written in the style of
\eqref{eq:9} as
\begin{equation}
n!\,u(x_1,\ldots,x_n)
=\sum_{k=0}^n(-1)^{n-k}\sum_{J;\,|J|=k}\tilde u(x_0+S_J).
\label{K2}
\end{equation}

Formula \eqref{K2} can be proved by slight adaptations of the
arguments leading above to \eqref{eq:9}. Just note that in
\eqref{eq:5} $\Delta _{x_n}\Delta_{x_{n-1}}\ldots\Delta _{x_1}\tilde
u$ is a constant function, by which it can be evaluated at any $x$ and
not necessarily at 0 following the definition of Tr. Thus
\eqref{eq:17} can be used for $x=x_0$.

When we put $x_0:=-\tfrac12(x_1+\cdots+x_n)$ (not allowing a field of
characteristic 2) in \eqref{K1} and use the homogeneity of $\tilde u$,
then we obtain
\begin{equation}
2^n n!\,u(x_1,\ldots,x_n)
=\sum_{\ep_1,\ldots,\ep_n=0}^1(-1)^{\ep_1+\cdots\ep_n}
\tilde u\big((-1)^{\ep_1} x_1+\cdots+(-1)^{\ep_n} x_n\big).
\label{K3}
\end{equation}
For $n=2$ this is \eqref{eq:1a} together with the fact that $\tilde u$
is an even function in this case. Note that \eqref{K2} and \eqref{K3}
have right-hand sides which are symmetric in $x_1,\ldots,x_n$, thus
giving a partial answer to the question in Section~5.

While the proof of \eqref{eq:9} as given by Thomas in Sections 2 and 3
is already a big advance as compared to the proof in \cite{mo34},
the proof of \eqref{K1} (and therefore of \eqref{K5} and
\eqref{eq:9}) as given in \cite{bs71} is extremely simple. Here is a sketch.
We start by expanding
\begin{multline*}
\sum_{\ep_1,\ldots,\ep_n=0}^1(-1)^{n-(\ep_1+\cdots\ep_n)}
\tilde u(\ep_0 x_0+\ep_1 x_1+\cdots+\ep_n x_n)\\
=\sum_{\ep_1,\ldots,\ep_n=0}^1(-1)^{n-(\ep_1+\cdots\ep_n)}
\sum_{i_1,\ldots,i_n=0}^n
u(\ep_{i_1}x_{i_1},\ldots,\ep_{i_n}x_{i_n}) \ .
\end{multline*}
Now any term in the inner sum on the right where some
$j\in\{1,\ldots,n\}$ is missing in $i_1,\ldots,i_n$ will be killed by
$\sum_{\ep_j=0}^1(-1)^{n-\ep_j}$. Hence, the inner sum on the right
only runs over permutations $i_1,\ldots,i_n$ of $1,\ldots,n$. Now use
symmetry of $u$ and the fact that $u$ vanishes as soon as one of its
arguments equals 0. Accordingly the right-hand side equals 
$n!\,u(x_1,\ldots,x_n)$. Thus, by putting $\ep_0=1$, we have proven
\eqref{K1}.

Formula \eqref{eq:9} has some resemblance to the inclusion-exclusion
principle; see, for instance, \cite[Theorem 1.6.1]{bala}. Let $A$ be a
finite set with subsets $A_1,\ldots,A_n$ and let $A_j'$ denote the
complement of $A_j$ in $A$. Then the inclusion-exclusion principle can
be formulated as
\begin{equation}
|A_1'\cap\ldots\cap A_n'|=\sum_{k=0}^n (-1)^k
\sum_{1\le i_1<i_2<\ldots<i_k\le n} |A_{i_1}\cap\ldots\cap A_{i_k}| \ .
\label{K6}
\end{equation}
Formula \eqref{K6} is an immediate consequence of
\begin{equation}
(1-\chi_{A_1})\ldots(1-\chi_{A_n})=
\sum_{k=0}^n(-1)^k
\sum_{1\le i_1<i_2<\ldots<i_k\le n}
\chi_{\raisemath{-1pt}A_{\raisemath{-1pt}i_{\raisemath{-1pt}1}}}
\ldots\chi_{\raisemath{-1pt}A_{\raisemath{-1pt}i_k}} \ .
\label{K7}
\end{equation}
Just evaluate both sides of \eqref{K7} at $x$ and sum over $x\in A$.
Formula \eqref{K7} in its turn is essentially the same as \eqref{eq:13},
working with commuting indeterminates $\chi_{A_i}$ in \eqref{K7}
and with commuting indeterminates $\sigma_i$ in \eqref{eq:12}.

\end{document}